\documentclass[9pt]{extarticle}
\usepackage{amsmath, graphicx}
\usepackage{amssymb}
\usepackage{amsfonts}
\usepackage{amscd}
\usepackage[usenames]{color}
\usepackage{amsthm}
\usepackage{tikz}
\usepackage[top=1.12in, bottom=1.27in, left=1.38in, right=1.38in]{geometry} 
\usepackage{layout}
\usepackage{makecell}
\usepackage{multirow}

\usepackage[hidelinks,draft=false]{hyperref}
\usepackage{bookmark}

\bookmarksetup{
  open,
  numbered,
  addtohook={%
    \ifnum\bookmarkget{level}>2 
      \renewcommand*{\numberline}[1]{}%
    \fi
  },
}

\usepackage{fancybox}
\usepackage{dsfont}
\usepackage{chngcntr}
\counterwithin*{equation}{section}
\usepackage{bm}
\usepackage{mathtools} 
\usepackage{stmaryrd} 
\usepackage{extarrows} 
\usepackage{cancel}

\input{xypic}
\xyoption{all}

\usepackage{threeparttable}
\usepackage{array}
\usepackage{booktabs}
\usepackage{blindtext}
\usepackage{mathrsfs}  

\usepackage{enumitem}
\usepackage{bbm} 
\usepackage{setspace} 
\usepackage{scalerel} 

\usepackage[nobottomtitles*]{titlesec}

\usepackage{nicematrix}

\newtheorem{theorem}{Theorem}[section]
\newtheorem{prop}[theorem]{Proposition}

\newtheorem{lemma}[theorem]{Lemma}
\newtheorem{cor}[theorem]{Corollary}

\newtheorem{defn}[theorem]{Definition}

\newtheorem{remark}[theorem]{Remark}
\newenvironment{rem}{\begin{remark}\rm}{\end{remark}}

\newtheorem{facts}[theorem]{Fact}

\newtheorem{example}[theorem]{Example}
\newenvironment{ex}{\begin{example}\rm}{\end{example}}

\newtheorem{exercise}[theorem]{Exercise}

\newtheorem{terminology}[theorem]{Terminology}

\newtheorem{notation}[theorem]{Notation}

\newtheorem{observation}[theorem]{Observation}

\newtheorem{question}[theorem]{Question}

\makeatletter
\DeclareRobustCommand*\uell{\mathpalette\@uell\relax}
\newcommand*\@uell[2]{
  \setbox0=\hbox{$#1\ell$}
  \setbox1=\hbox{\rotatebox{15}{$#1\ell$}}
  \dimen0=\wd0 \advance\dimen0 by -\wd1 \divide\dimen0 by 2
  \mathord{\lower 0.1ex \hbox{\kern\dimen0\unhbox1\kern\dimen0}}
}

\makeatletter
\renewcommand*\env@matrix[1][*\c@MaxMatrixCols c]{%
  \hskip -\arraycolsep
  \let\@ifnextchar\new@ifnextchar
  \array{#1}}
\makeatother
\usepackage[normalem]{ulem} 

\usepackage{adjustbox}

\def\a{\alpha}
\def\b{\beta}
\def\ga{\gamma}
\def\d{\delta}

\def\s{\sigma}

\def\o{\omega}
\def\ep{\epsilon}

\def\vp{\varphi}

\def\L{\Lambda}
\def\Ga{\Gamma}
\def\Om{\Omega}

\def\C{\mathbb C}

\def\R{\mathbb R}

\def\S{\mathbb S}

\def\rank{{\rm rank}}

\def\lb{\langle}
\def\rb{\rangle}

\def\xr{\xrightarrow}

\def\ov{\overline}

\def\non{\noindent}

\def\p{\prime}

\def\spinc{ {\rm{spin}^{{\rm c}} }}

\def\End{{\rm End}}

\def\Sym{{\rm Sym}}

\def\c{{ \rm c}}

\def\Id{{\mathsf{Id}}}

\def\ind{{\rm ind\, }}
\def\tr{{\rm tr }}

\def\ch{{\rm ch}}

\def\del{\bar{\partial}}

\def\ov{\overline}

\def\bu{$\bullet$}

\newcommand{\sdfrac}[2]{\mbox{\small$\displaystyle\frac{#1}{#2}$}}

\def\w{\wedge}

\def\ch{ {\rm ch} }

\usepackage{scalerel,stackengine}

\def\Id{{\rm Id}}
\usepackage{graphicx}

\title{{\bf Note on concentration via the conjugate-linear Hodge star operator}}

\author{Junho Lee}

\date{\empty}
\begin{document}

\maketitle

\begin{abstract}

We construct conjugate-linear perturbations of twisted $\spinc$ Dirac operators
on compact almost hermitian manifolds of dimension congruent to $2$ or $6$ modulo $8$,
employing the conjugate-linear Hodge star operator rescaled by unit complex numbers depending on degree.
These perturbations satisfy the concentration principle.
\end{abstract}

\section{Introduction}

One of the most fruitful ideas in geometry is {\em localization}, which underlies the celebrated Witten deformation. Based on this idea, consider a first-order elliptic operator
$$
D:\Ga(E)\to \Ga(F)
$$
on a compact Riemannian manifold $X$, and a deformation of the form
$$
D_s=D+sA,
$$
where $s\in\R$ and $A:E\to F$ is a bundle map.

As shown in \cite{Ma}, if the perturbation $A$ satisfies the following algebraic condition:
\begin{equation}\label{E:CP}
\s_{D^*}(\ga)\circ A + A^*\circ \s_D(\ga)=0\quad \forall \ga\in T^*X,
\end{equation}
where $\s_D$ denotes the principal symbol of $D$, then as $s\to \infty$, the kernel of $D_s$ (as well as low eigenspaces of $D_s^*D_s$) becomes concentrated near the {\em singular set} $Z_A$ of $A$, defined by
\begin{equation*}
Z_A:=\{ x\in X: \ker(A_x)\ne 0\}.
\end{equation*}

The most well-known example of this {\em concentration principle} is the Witten deformation~\cite{W}, where $D$ is the Hodge--de Rham operator.
Taubes' localization proof of the Riemann--Roch theorem is another well-known and ingenious example~\cite{T}.

When $D$ is a Dirac operator, Prokhorenkov and Richardson~\cite{PR} formulated the concentration principle for complex-linear perturbations $A$ and classified such perturbations. Maridakis~\cite{Ma} later extended the principle to include real elliptic operators as described above, yielding many interesting examples.

The concentration principle has appeared in various contexts; see, for example, \cite{R, LP, CC, P, Na, CL}. In particular, Nagy~\cite{Na}
constructed conjugate-linear perturbations $A$ satisfying the concentration principle in a general setting, using nondegenerate invariant bilinear forms on spinor modules.

The present work aims to generalize the conjugate-linear perturbations introduced in \cite{LP} (see Remark~\ref{rem}) by constructing
conjugate-linear perturbations $A_\vp$ of the twisted $\spinc$ Dirac operator
\begin{equation}\label{E:ITD}
D_E:\Ga(\S^+\otimes E)\to\Ga(\S^-\otimes E)
\end{equation}
on a compact almost hermitian manifold $X$ with $\dim X\equiv 2, 6\pmod{8}$. These perturbations are defined via the  conjugate-linear Hodge star operator rescaled by unit complex numbers depending on degree, together with a section $\vp$ of $K_X\otimes (E^*)^2$, and satisfy the condition (\ref{E:CP}) (see Theorem~\ref{T:Main}).  The dimension restriction is due to $X$ being an almost complex manifold and the fact that the perturbations $A_\vp$ interchange $\S^+$ and $\S^-$ when $\dim X\equiv 2, 6\pmod{8}$.

The rest of the paper is organized as follows.
In Section~2, we briefly review basic well-known facts about twisted $\spinc$ Dirac operators $D_E$, and present observations central to our construction. In Section~3, we construct conjugate-linear perturbations $A_\vp$ of the operator $D_E$ satisfying (\ref{E:CP}), and provide examples where the singular set $Z_{A_\vp}$ is empty; in such cases, the index of the operator $D_E$ is zero by the concentration principle (see Corollary~\ref{C:ES}).

\section{Preliminaries}

Throughout this paper, $X$ denotes a compact almost hermitian manifold of dimension $2n$, equipped with an almost complex structure $J$ and a Riemannian metric $g$ compatible with the almost complex structure $J$, while $E$ denotes a hermitian vector bundle over $X$.

This section briefly reviews basis well-known facts relevant to our discussion, fixes notation, and presents key observations to our construction.

\subsection{Twisted Dirac operator}

The almost complex structure $J$, extended by complex linearity to $TX\otimes\C$ and $T^*X\otimes\C$, induces the decompositions:
$$
TX\otimes\C=(TX)^{1,0}\oplus (TX)^{0,1}\quad\text{and}\quad
T^*X\otimes\C=(T^*X)^{1,0}\oplus (T^*X)^{0,1},
$$
where $(TX)^{1,0}$ and $(TX)^{0,1}$ are the $\pm i$-eigenbundles of $J$, and similarly for $T^*X\otimes \C$.

Let $g_\C$ denote the complex bilinear extension of the metric $g$ to $\L^*_\C X=\L^*X\otimes\C$. The associated hermitian metric on $\L^*_\C X$ is defined by
$$
\lb \a,\b\rb := g_\C(\a,\bar{\b}).
$$

The almost hermitian manifold $X$ carries a canonical $\spinc$ structure, which comes with a spinor bundle
\begin{equation}\label{E:CG}
\S=\S^+\oplus\S^-\quad\text{where}\quad\S^\pm=\Lambda^{0,\text{even/odd}}X,
\end{equation}
and a Clifford multiplication $\c:T^*X\to\End_\C(\S^\pm,\S^\mp)$, given by
\begin{equation}\label{D:CliffM}
\c(\ga)\a:=\sqrt{2}\big(\ga^{0,1}\w\a \ -\iota(\ga^{1,0})\a\big),
\end{equation}
where $\iota(\ga^{1,0})$ denotes
contraction by $\ga^\sharp_{0,1}$, the vector dual to $\ga^{1,0}$ with respect to $g_\C$.
Note that
$$
\lb \ga^{0,1}\w \a,\b\rb = \lb \a,\iota(\ga^{1,0})\b\rb\quad\text{and}\quad
\lb\c(\ga)\a,\b\rb=-\lb\a,\c(\ga)\b\rb.
$$

A $\spinc$ connection $\nabla$ on the spinor bundle $\S$, induced from the Levi-Civita connection on $X$ and a hermitian connection on $K_X^{-1}$, preserves the chiral grading in (\ref{E:CG}).
The induced connection $\nabla^+$ on $\S^+$, together with a hermitian connection $\nabla^E$ on the vector bundle $E$ and the Clifford multiplication,
defines the twisted Dirac operator (\ref{E:ITD}) via the composition
\begin{equation}\label{E:TD}
D_E:\Ga(\S^+\otimes E)\xr{\nabla^+\otimes\Id_E+\Id_{\S^+}\otimes\nabla^E}\Ga(T^*X\otimes \S^+\otimes E)\xr{\c\otimes \Id_E}\Ga(\S^-\otimes E).
\end{equation}
Consequently, the principal symbols of $D_E$ and its adjoint $D_E^*$ are given by
$$
\s_{D_E}(\ga)=\c(\ga)\otimes \Id_E\quad \text{and}\quad
\s_{D_E^*}(\ga)=-\s_D(\ga)^*=\c(\ga)\otimes \Id_E
$$
for all $\ga\in T^*X$. (See, e.g., \cite{Ni, S} for more on Dirac operators.)

\subsection{Conjugate-linear Hodge star operator}

The conjugate-linear operator $\bar{*}$, which maps $\L^{p,q}X$ to $\L^{n-p,n-q}X$,  is defined by
$$
\a\w \bar{*}(\b)=\lb \a,\b\rb dv_g
$$
for $\a,\b\in \Om^{p,q}(X)$, where $dv_g$ denotes the volume form on $X$.
The following lemma is a key observation in our construction.

\begin{lemma}\label{L:CLHS}
For any $\b\in\Om^{0,p}(X)$, $\ga\in\Om^1(X)$, and $\eta\in\Om^{n,0}(X)$, we have:
\begin{itemize}
\item[(a)]
$|\eta|^2\,\bar{*}\big(\ga^{0,1}\w \b\big)=(-1)^{n(p+1)+p}\,\eta\w \iota(\ga^{1,0})\bar{*}\big(\eta\w \b\big)$.

\item[(b)]
$|\eta|^2\,\bar{*}\big(\iota(\ga^{1,0})\b\big) = (-1)^{(n+1)(p-1)}\,\eta\w\ga^{0,1}\w \bar{*}\big(\eta\w \b\big)$.
\end{itemize}

\end{lemma}

\begin{proof}
For any $\a\in\Om^{0,p+1}X$, we have:
\begin{align*}
\a\w |\eta|^2\,\bar{*}\big(\ga^{0,1}\w \b\big)
&=|\eta|^2\big\lb \a,\ga^{0,1}\w \b\big\rb dv_g
=|\eta|^2\,\big\lb \iota(\ga^{1,0})\a, \b\big\rb dv_g
\\
&=\big\lb \eta\w \iota(\ga^{1,0})\a,\eta\w \b\big\rb dv_g
=\eta\w \iota(\ga^{1,0})\a\w \bar{*}\big(\eta\w \b\big)
\\
&=(-1)^p\,\eta \w \a\w \iota(\ga^{1,0})\bar{*}\big(\eta\w \b\big)
\\
&=\a\w\Big((-1)^{n(p+1)+p}\,\eta\w \iota(\ga^{1,0})\bar{*}\big(\eta\w \b\big)\Big),
\end{align*}
where the third equality follows from the definition of the hermitian metric on $\L^*_\C X$,
and the sign factors follow since contraction and wedge product are antiderivations. This implies (a).
The proof of (b) is similar. $\qedhere$
\end{proof}

\subsubsection{Rescaling}

For our purposes,  we rescale the operator $\bar{*}$ as follows:
\begin{equation}\label{D:HSWS}
\tau:=\ep_{k}\bar{*}\quad \text{on}\quad \L^{k}_\C X, \quad \text{where}\quad \ep_{k}=i^{k(k-1)+n}.
\end{equation}
The unit complex number $\ep_k$, which also appears in the context of the signature operator, plays a useful role in our construction. Observe that
\begin{equation}\label{E:Square}
\bar{*}^2=(-1)^{k}\Id\quad\text{on}\quad \L^{k}_\C X\qquad\text{and}\qquad\tau^2=(-1)^n\Id.
\end{equation}
The following identity is central to our construction.

\begin{prop}\label{P:CLO}
For any $\b\in \Om^{0,p}(X)$, $\ga\in \Om^1(X)$, and $\eta\in\Om^{n,0}(X)$, we have
$$
|\eta|^2\,\tau\circ\c(\ga)(\b) = (-1)^{\frac{n(n+1)}{2}+1}\,\eta\w \c(\ga)\circ \tau(\eta\w\b).
$$
\end{prop}

\begin{proof}
The proof follows from Lemma~\ref{L:CLHS} and the identity:
$$
\ep_{p+1}\ep_{n+p}^{-1}(-1)^{n(p+1)+p}=(-1)^{\frac{n(n+1)}{2}}=\ep_{p-1}\ep_{n+p}^{-1}(-1)^{(n+1)(p-1)}.
\qedhere
$$
\end{proof}

\subsubsection{Adjoint}

Since $\tau$ is conjugate-linear, its adjoint $\tau^*$ is defined with respect to the real part of the hermitian metric.
By (\ref{E:Square}), we obtain:

\begin{lemma}\label{L:RA}
$\tau^*=(-1)^{n}\tau$.
\end{lemma}

\begin{proof}
For any $\a\in\Om^k_\C(X)$ and $\b\in\Om^{2n-k}_\C(X)$, we have
\begin{align*}
\lb \b, \tau\a\rb dv_g &=\ov{\ep}_k(-1)^k\,\b\w \a
= \ov{\ep}_k\,\a\w\b =\ov{\ep}_k\,\ep_{2n-k}(-1)^{2n-k}\,\a\w \bar{*}(\tau\b)
\\
&=(-1)^{n}\lb \a,\tau\b\rb dv_g.
\end{align*}
This shows that ${\rm Re}\lb \tau\a,\b\rb
={\rm Re}\lb\a,(-1)^n\tau\b\rb$, and hence proves the lemma.
$\qedhere$
\end{proof}

\subsubsection{Operator with values in $E$}

The hermitian metric on $E$, still denoted by $\lb\ ,\ \rb$, defines a conjugate-linear isomorphism
$E\to E^*$ by $\xi\mapsto \xi^*:=\lb\ ,\xi\rb$.
This induces the map
$$
\bar{*}_E:\L^{p,q}X\otimes E^*\to \L^{n-p,n-q}X\otimes E
$$
defined by
$\bar{*}_E(\a\otimes\xi^*)=\bar{*}(\a)\otimes \xi$.
Now we set
$$
\tau_E:=\ep_k\bar{*}_E\quad \text{on}\quad \L^{k}_\C X\otimes E.
$$
Note that $\tau_E(\a\otimes\xi^*)=\tau(\a)\otimes \xi$ and $\tau_E^*(\a\otimes\xi)=(-1)^n\tau(\a)\otimes \xi^*$.

\section{Conjugate-linear perturbations}

\subsection{Concentrating condition}

This subsection constructs conjugate-linear perturbations of the twisted Dirac operator $D$ in (\ref{E:TD}) that satisfy
the concentrating condition (\ref{E:CP}).

\begin{defn}\label{D:CLP}
Given a section $\vp\in\Ga\big(K_X\otimes (E^*)^2\big)$, viewed as a complex-linear bundle map
$\vp:E\to K_X\otimes E^*$, we define a conjugate-linear bundle map
$$
A_\vp:\L^{0,p}X\otimes E\to \L^{0,n-p}X\otimes E
$$
by the composition
$$
A_\vp:\L^{0,p}\otimes E \xr{\Id\otimes \vp} \L^{0,p}\otimes K_X\otimes E^* \cong \L^{n,p}\otimes E^*
\xr{\ \tau_E\ }\L^{0,n-p}\otimes E.
$$

\end{defn}

\medskip

Noting $A_\vp:\S^\pm\otimes E\to \S^\mp\otimes E$ when $\dim X\equiv 2, 6\pmod{8}$, and using the decomposition
$$
E^*\otimes E^* =\Sym^2 E^*\oplus \L^2 E^*,
$$
we obtain concentrating pairs $(\s_{D_E},A_\vp)$ in the sense of \cite{Ma}:

\begin{theorem}\label{T:Main}
Let $D_E$ be the twisted Dirac operator in (\ref{E:TD}) and $A_\vp$  as in Definition~\ref{D:CLP}. Suppose
$\vp\in\Ga(K_X\otimes \Sym^2 E^*)$ if $\dim X\equiv 2\pmod{8}$ and $\vp\in\Ga(K_X\otimes\L^2E^*)$
if $\dim X\equiv 6\pmod{8}$. Then $A_\vp$ satisfies the concentrating condition
\begin{equation}\label{E:MainCC}
\s_{D_E^*}(\ga)\circ A_\vp + A_\vp^*\circ \s_{D_E}(\ga)=0\quad\forall\ga\in T^*X.
\end{equation}
\end{theorem}

\begin{proof}
In unitary frames $\{\phi_i\}$ for $E$ (with dual coframe $\{\phi^i\}$) and $\eta$ for $K_X$ (with dual coframe $\eta^*$),
the complex bundle map $\vp:E\to K_X\otimes E^*$ is given by
\begin{equation}\label{E:LE}
\vp=\vp_{ij} \eta\otimes \phi^i\otimes\phi^j:\xi=\xi_k\phi_k\mapsto \vp_{ij}\xi_j\eta\otimes\phi^i.
\end{equation}
It follows that for any $\b\otimes \xi\in \Ga(\S^+\otimes E)$ and any $\ga\in \Om^1(X)$,
\begin{align}\label{A:E1}
\s_{D^*}(\ga)\circ A_\vp(\b\otimes\xi)
&=(\c(\ga)\otimes \Id_E)\circ \tau_E\big(\eta\w\b\otimes \vp_{ij}\xi_j\phi^i\big)
=\c(\ga)\tau(\eta\w \b)\otimes \ov{\vp}_{ij}\ov{\xi}_j\phi_i.
\end{align}

On the other hand, since $|\eta|^2=1$, by Lemma~\ref{L:RA} and Proposition~\ref{P:CLO}, we have
\begin{align}\label{A:E2}
\tau_E^*\circ\s_D(\ga)(\b\otimes\xi)
&=(-1)^n\tau(\c(\ga)\b) \otimes \ov{\xi}_k\phi^k
=(-1)^{\frac{n(n+1)}{2}+n+1}\,\eta\w \c(\ga)\circ \tau(\eta\w\b)\otimes \ov{\xi}_k\phi^k.
\end{align}
Observe that the adjoint $\vp^*:K_X\otimes E^*\to E$ of the complex bundle map $\vp$ is given by
$$
\vp^*=\ov{\vp}_{ji}\phi_i\otimes \eta^*\otimes\phi_j:\nu=\nu_k\eta\otimes \phi^k\mapsto \ov{\vp}_{ji}\nu_j\phi_i.
$$
With the canonical isomorphism $\L^{n,p}X\cong \L^{0,p}\otimes K_X$, from (\ref{A:E2}) we obtain
\begin{align}\label{A:E3}
A_\vp^*\circ\s_D(\ga)(\b\otimes\xi)
&=(-1)^{\frac{n(n+1)}{2}+n+1}\,\Id\otimes\vp^*\big(\c(\ga)\circ \tau(\eta\w\b)\otimes\ov{\xi}_k\eta\otimes\phi^k\big)
\notag\\
&=(-1)^{\frac{n(n+1)}{2}+n+1}\,\c(\ga)\circ \tau(\eta\w\b)\otimes \ov{\vp}_{ji}\ov{\xi}_j\phi_i.
\end{align}

By the assumption on $\vp$, (\ref{A:E1}), and (\ref{A:E3}), the proof now follows from the identity
$$
\sdfrac{n(n+1)}{2}+n+1=\left\{
\begin{array}{ll}
1\ \ \ &\text{if $n\equiv 1\pmod{4}$}
\\
0\ \ \ &\text{if $n\equiv 3\pmod{4}$}
\end{array}
\right. \qedhere
$$
\end{proof}

\medskip

The next two remarks are related to the motivation for our approach.

\begin{rem}
For the twisted Dirac operator $D_E:\Ga(\S^+\otimes E)\to \Ga(\S^-\otimes E)$, there is no complex-linear bundle map
$A:\S^+\otimes E\to \S^-\otimes E$ that satisfies the concentrating condition (\ref{E:MainCC}). This explains why the Dolbeault and signature
operators do not admit such complex perturbations. (See Section~2 of \cite{PR}.)

\end{rem}

\begin{rem}\label{rem}
When $\dim X=2$, the manifold $X$ is a compact complex curve, and the operator $D_E=\sqrt{2}\,\del_E:\Om^0(E)\to\Omega^{0,1}(E)$ is the usual Cauchy--Riemann operator on $E$. In \cite{LP}, when $E$ is a holomorphic line bundle and $\vp$ is a holomorphic section of $K_X\otimes (E^*)^2$, a conjugate-linear bundle map $R:E\to \L^{0,1}X\otimes E$ is defined as the composition $R=\bar{*}_E\circ \vp$. In this case,
$A_\vp=i R$ and $A^*_\vp=i R^*$, so our construction generalizes the conjugate-linear bundle map $R$.

Indeed,  $A_\vp=iR$ satisfies the following equation, which is stronger than the concentrating condition:
$$
D_E^*\circ A_\vp + A_\vp^*\circ D_E =0.
$$
(See Lemma~2.1 of \cite{LP}.) When $E$ is a theta characteristic (i.e., $E^2\cong K_X$), this implies that $A_\vp$ restricts to a conjugate-linear isomorphism $\ker D_E\to \ker D_E^*$, which leads to a proof of the Atiyah--Mumford theorem:  $h^0(E)\,(=\dim \ker D_E)$ mod 2 is {\em deformation invariant} (see Section~3 of \cite{LP}).
\end{rem}

\subsection{Concentration principle}

Let  $D_E, A_\vp:\Ga(\S^+\otimes E)\to \Ga(\S^-\otimes E)$ be as in Theorem~\ref{T:Main} and consider the deformation of the twisted Dirac operator $D_E$
given by
$$
D_s=D_E+sA_\vp.
$$
The following concentration principle  shows that as $s\to \infty$, the kernel of $D_s$ (as well as low eigenspaces of $D_s^*D_s$) becomes concentrated near the singular set
$$
Z_\vp:=Z_{A_\vp}=\{x\in X:\det(\vp_x)=0\}.
$$

\begin{theorem} \label{T:CP}
Let $D_s=D_E+A_\vp$ be as above. For each $\d>0$ and $C\geq 0$, there exits a constant $C^\p=C^\p(\d,A_\vp,C)>0$ such that if $\zeta\in \Ga(\S^+\otimes E)$ satisfies $||\zeta||_2=1$ and $||D_s\zeta||_2^2\leq C|s|$,
where $||\cdot||_2$ denotes the $L^2$-norm, then we have
$$
\int_{X\setminus Z_\vp(\d)}|\zeta|^2 dv_g < \sdfrac{C^\prime}{|s|},
$$
where $Z_\vp(\d)$ is the $\d$-neighborhood of the singular set $Z_\vp$.
\end{theorem}

\begin{proof}
Since $A_\vp$ satisfies the concentrating condition (\ref{E:MainCC}), the theorem follows directly from Proposition~2.4 of \cite{Ma}. $\qedhere$

\end{proof}

In applying this principle, the singular set plays a crucial role. If $Z_\vp=X$, no concentration occurs. For example, this happens when
$$
\dim X\equiv 6\pmod{8}\qquad \text{and}\qquad  \text{$\rank(E)$ is odd},
$$
since $\vp_{ij}=-\vp_{ji}$ in (\ref{E:LE}), and thus $\det(\vp_{ij})=0$. On the other hand, when $Z_\vp=\emptyset$, we have:

\begin{cor}\label{C:ES}
Let $D_E$ and $A_\vp$ be as in Theorem~\ref{T:Main}.
If $Z_\vp=\emptyset$, then $\ind D_E=0$.
\end{cor}

\begin{proof}
By Theorem~\ref{T:CP},  the singular set $Z_\vp=\emptyset$ implies that for sufficiently large $s$, the operator
$D_s$ has trivial kernel. It follows that
$$
\ind D_E=\tfrac12\ind_{\hspace{-1pt}\R}D_s\leq 0,
$$
where the second term denotes the index of $D_s$ as a real operator, and the equality follows
since $D_s$ is a compact perturbation of $D_E$.
Applying the same argument to the operator $D_E^*+sA_\vp$ yields
$\ind D_E^*=-\ind D_E \leq 0$, so the claim follows. $\qedhere$
\end{proof}


\subsubsection{Compact spin almost hermitan manifolds}

Suppose $c_1(X)\equiv 0\pmod{2}$, so that $X$ admits a spin structure.  Fix a spin structure $\s$ on $X$, which is equivalent to
choosing a square root $N_\s$ of the canonical bundle $K_X$, i.e., $N_\s^2\cong K_X$. In this case, the complex spinor bundle
$\S_\s$ of the spin structure $\s$ is given by
$$
\S_\s=\S\otimes N_\s.
$$

Let $E=N_\s\otimes F$, where $F$ is a hermitian vector bundle $F$ over $X$.
A section $\psi\in\Ga\big((F^*)^2\big)$ induces a section
$\vp\in\Ga\big(K_X\otimes (E^*)^2\big)$ via the isomorphism
$(F^*)^2\cong K_X\otimes (E^*)^2$,
which restricts to yield
$$
\Sym^2F^*\cong K_X\otimes \Sym^2E^*\quad\text{and}\quad\L^2F^*\cong K_X\otimes\L^2E^*.
$$
When $\dim X\equiv 2,6\pmod{8}$, we set
$$
A_\psi:=A_\vp:\S^\pm\otimes E=\S^\pm_\s\otimes F\to \S^\mp\otimes E=\S_\s^\mp\otimes F,
$$
where $A_\vp$ is the conjugate-linear bundle map as in Theorem~\ref{T:Main}. The singular set of this bundle map is then given by
$$
Z_\psi:=Z_\vp=\{x\in X:\det(\psi_x)=0\}.
$$

Below are examples of bundle maps $A_\psi$ whose singular set $Z_\psi=\emptyset$.
\begin{itemize}
\item[(a)]
Let $\dim X\equiv 2\pmod{8}$. The following are typical examples of nondegenerate symmetric complex bilinear forms
$\psi\in\Ga(\Sym^2F^*)$ on $F$:
\begin{itemize}
\item[\bu]
$F=TX\otimes \C$ with $\psi=g_\C$, the complex bilinear extension of the metric $g$ to $TX\otimes\C$, or
\item[\bu]
$F=\End_\C(W)$, where $W$ is a hermitian vector bundle over $X$, and $\psi$ is the natural trace pairing given by $\psi(A,B)=\tr(A\circ B)$.
\end{itemize}
\item[(b)]
Let $\dim X\equiv 6\pmod{8}$. A complex vector bundle $F$ is called a {\em complex symplectic vector bundle} if it is equipped with  a nondegenerate skew-symmetric complex bilinear form $\psi\in\Ga(\L^2F^*)$. Standard examples include:
\begin{itemize}
\item[\bu]
$F=TX\otimes\C$ with $\psi=\o_\C$, where $\o_\C(u,v)=g_\C(u,Jv)$, or
\item[\bu]
$F=W\oplus W^*$ with $\psi$ defined by $\psi\big((w_1,f_1),(w_2,f_2)\big)=f_2(w_1)-f_1(w_2)$.
\end{itemize}
\end{itemize}
In all of the above examples $(F,\psi)$, the bundle map $A_\psi\,(=A_\vp)$ satisfies the concentrating condition, and
the singular set $Z_\psi=\emptyset$. Hence,
by Corollary~\ref{C:ES}, the twisted Dirac operator
\begin{equation*}\label{E:TSD}
D_{N_\s\otimes F}:\Ga(\S_\s^+\otimes F)\to \Ga(\S_\s^-\otimes F)
\end{equation*}
has index $\ind  D_{N_\s\otimes F}=0$.

\begin{rem}
The vanishing of the index  also follows easily from the Atiyah-Singer index theorem:
$$
\ind D_{N_\s\otimes F} =\int_X \ch(F)\hat{{\rm A}}(X),
$$
where $\hat{{\rm A}}(X)$ is a polynomial in the  Pontryagin classes $p_k(X)=(-1)^kc_{2k}(X)$ and $\ch(F)$ is a polynomial in the even Chern classes $c_{2\ell}(F)$ (since $F\cong F^*$ as a complex vector bundle). As $\frac12 \dim X$ is odd, there is no top-degree term in the integrand, and therefore the integral vanishes.

\end{rem}


\medskip

\non
Department of  Mathematics,  University of Central Florida, Orlando, FL 32816

\non
Email: junho.lee\@@ucf.edu


\begin{thebibliography}{99}



\bibitem{CC} C. Gerig and C. Wendl, {\em Generic transversality for unbranched covers of closed pseudoholomorphic curves},
Comm. Pure Appl. Math. 70 (2017), no. 3, 409–443.

\bibitem{CL} K. Choi and J. Lee, {\em Witten deformation and divergence-free symmetric Killing 2-tensors}, arXiv:2405.10520.






\bibitem{LP} J. Lee and T. Parker, {\em  Spin Hurwitz numbers and the Gromov-Witten invariants of Kähler surfaces},
Comm. Anal. Geom. 21 (2013), no. 5, 1015–1060.

\bibitem{Ma} M. Maridakis, {\em Spinor pairs and the concentrating principle for Dirac operators},
Trans. Amer. Math. Soc. {\bf 369} (2017), no. 3, 2231--2254.



\bibitem{Na} A. Nagy, {\em  Conjugate linear perturbations of Dirac operators and Majorana fermions},
J. Geom. Anal. {\bf 35} (2025), no. 6, Paper No. 169.



\bibitem{Ni} L. Nicolaescu, {\em Notes on Seiberg-Witten theory}, Grad. Stud. Math., {\bf 28}
American Mathematical Society, Providence, RI, 2000.







\bibitem{P} G. Parker, {\em Concentrating Dirac operators and generalized Seiberg-Witten equations}, arXive:2307.00694.



\bibitem{PR} I. Prokhorenkov and K. Richardson, {\em Perturbations of Dirac operators},
J. Geom. Phys. {\bf 57} (2006), no. 1, 297--321.

\bibitem{R} D. Rauch, {\em Perturbations of the d-bar operator}, Doctoral dissertatiion, Harvard University, 2004.



\bibitem{S} D. Salamon, {\em Spin Geometry and Seiberg-Witten Invariants}, unpublished manuscript, 1999.











\bibitem{T} C. Taubes, {\em Counting pseudo-holomorphic submanifolds in dimension 4}, J. Diff. Geom., {\bf 44} (1996),
no. 4, 818--893.


\bibitem{W} Edward Witten, {\em Supersymmetry and Morse theory}, J. Diff. Deom.,
{\bf 17} (1982), no. 4,  661--692.



\end{thebibliography}
\end{document}